\documentclass[11pt,leqno]{amsart}

\usepackage{amsmath,amsthm}
\usepackage {latexsym}
\usepackage{amssymb}

\newcommand{\eps}{{\varepsilon}}
\newcommand{\R}{{\mathbb R}}

\newcommand{\Compl}{{\mathbb C}}

\newcommand{\X}{\overline{X}}

\newcommand{\les}{\lesssim}

\newcommand{\txt}{\textstyle}
\newcommand{\scr}{\scriptstyle}

\newcommand{\Lap}{\Delta}

\newcommand{\Res}{R_0^+}

\newcommand{\vp}{\varphi}

\newcommand{\la}{\langle}
\newcommand{\ra}{\rangle}
\newcommand{\1}{{\mathbf 1}}
\def\norm[#1][#2]{\|#1\|_{#2}}
\def\bignorm[#1][#2]{\big\|#1\big\|_{#2}}
\def\Bignorm[#1][#2]{\Big\|#1\Big\|_{#2}}
\def\japanese[#1]{\langle #1 \rangle}
\def\Im[#1]{{\rm Im}(#1)}
\def\Re[#1]{{\rm Re}(#1)}

\newtheorem{theorem}{Theorem}
\newtheorem{lemma}[theorem]{Lemma}

\newtheorem{corollary}[theorem]{Corollary}

\newtheorem{proposition}[theorem]{Proposition}

\theoremstyle{remark}
\newtheorem{remark}{Remark}

\begin{document}

\title[Dispersive Bound for Schr\"odinger Operators with Eigenvalues]{A dispersive bound for three-dimensional Schr\"odinger operators with zero energy
eigenvalues}
\date{September 20, 2008}

\author{Michael\ Goldberg}
\thanks{The author received partial support from NSF grant DMS-0600925 during
the preparation of this work.}
\address{Department of Mathematics, Johns Hopkins University,
3400 N. Charles St., Baltimore, MD 21218}
\email{mikeg@math.jhu.edu}

\begin{abstract}
We prove a dispersive estimate for the evolution of Schr\"odinger operators
$H = -\Delta + V(x)$ in ${\mathbb R}^3$.  The potential is allowed to be
a complex-valued function belonging to $L^p(\R^3)\cap L^q(\R^3)$,
$p < \frac32 < q$, so that $H$ need not be self-adjoint or even symmetric.
Some additional spectral conditions are imposed, namely that no resonances
of $H$ exist anywhere within the interval $[0,\infty)$ and that eigenfunctions
at zero (including generalized eigenfunctions) decay rapidly enough to be
integrable.  
\end{abstract}

\maketitle

\section{Introduction} \label{sec:intro}
The Schr\"odinger equation is representative of a larger class of
dispersive evolution equations, in which wave packets that are localized to
distinct regions of Fourier space propagate with different group velocity.
One consequence is that mass concentration tends to be transient in nature
as the 
frequency components of a solution eventually separate from one another.
Solutions of a dispersive equation are therefore expected to display 
an improvement in local regularity and/or a decrease in overall size over
long time periods.  We will concentrate here on evolution estimates that use
the $L^1(\R^3)$ norm of initial data to control the supremum of the solution
at later times.  As with any map from $L^1$ to $L^\infty$, this can
also be understood as a pointwise bound on the propagator kernel.

The free Schr\"odinger equation propagates forward to time $t$ through the
action of $e^{it\Delta}$ on the initial data. By standard Fourier inversion
identities, this is equivalent to convolution against
the complex Gaussian kernel $(-4\pi it)^{-3/2} e^{i(|x|^2/4t)}$. 
It immediately follows
that the free evolution satisfies a dispersive bound 
\begin{equation} \label{eq:freedispersive}
\norm[e^{it\Lap}f][\infty] \le (4\pi|t|)^{-3/2}\norm[f][1]
\end{equation}
 at all times $t \not= 0$.
In this paper we seek to prove similar estimates for the time evolution
$e^{-itH}$ induced by a perturbed Hamiltonian $H= -\Delta + V(x)$.
We assume that $V \in L^p(\R^3) \cap L^q(\R^3)$ for a pair of exponents
$p < \frac32 < q$, corresponding to homogeneous behavior on the order
of $|x|^{-2+\eps}$ for local singularities and $|x|^{-2-\eps}$
as $|x| \to \infty$.  There are no further restrictions on the size of $V$
or on its negative or imaginary parts.  In particular, there is no assurance
that $H$ is a symmetric or self-adjoint operator.

Eigenvalues of $H$ present an immediate obstruction to the validity of 
dispersive estimates.  If there exists a nonzero function $\Psi \in L^1(\R^3)$
that solves the eigenvalue equation $H\Psi = \lambda\Psi$, then the associated
Schr\"odinger evolution
\begin{equation*}
e^{-itH}\Psi = e^{-it\lambda}\Psi
\end{equation*}
certainly lacks the decay properties
of~\eqref{eq:freedispersive}.  Indeed, such a solution either fails to
decay at all (if $\lambda \in\R$) or experiences exponential growth according
to the imaginary part of $\lambda$.
Generalized eigenfunctions (i.e. solutions to $(H-\lambda)^k\Psi = 0$ for
some $k>1$) create a similar problem.  The series expansion of 
$e^{-itH}$ in powers of $t(H-\lambda)$ shows that the evolution of initial
data $\Psi$ must experience growth at a rate of $|e^{-it\lambda}|$ times a
degree $(k-1)$ polynomial in $t$.

These bound states can often be avoided by introducing the appropriate spectral
projection.  A revised dispersive estimate for $H = -\Delta + V$ might
take the form
\begin{equation} \label{eq:dispersive1}
\big\|e^{-itH}(I- {\txt \sum}_j P_{\lambda_j}(H))f\big\|_\infty 
  \les |t|^{-\frac32}\norm[f][1]
\end{equation}
where each $P_{\lambda_j}(H)$ is a projection onto the point spectrum of $H$
at the eigenvalue $\lambda_j \in \Compl$.  We use the term ``projection'' here
to indicate a bounded linear operator satisfying $P^2 = P$, though not
necessarily the canonical $L^2$-orthogonal projection.  In general the
correct choice for
$P_\lambda(H)$ will depend on the eigenfunctions of $H^*$ as well as those
of $H$.  Details of its construction are given in Section~\ref{sec:decomp}. 

One additional concern here is the possible existence of resonances, which
are solutions to the eigenvalue equation that do not decay rapidly enough 
to belong to $L^2(\R^3)$.  Resonances exhibit enough persistence behavior
(by virtue of their resemblance to $L^2$ bound states) to negate most
dispersive estimates, but they cannot be so easily removed with a 
spectral projection.

Our main theorem proves that~\eqref{eq:dispersive1} remains valid in the
presence of an eigenvalue at zero, so long as each of the eigenfunctions (and
generalized eigenfunctions) belongs to $L^1(\R^3)$.  Resonances must still
be forbidden.  To make a precise
statement we introduce a classification system for the eigenspace of $H$
lying over $\lambda = 0$.

Let $X_1$ represent the space of threshold eigenfunctions and resonances of
$H$, which are distributional solutions of $(-\Delta + V)\Psi = 0$ belonging
to $L^3_{\rm weak}(\R^3)$, the same class as the Green's function of
the Laplacian.  Among weighted $L^2$ spaces, this places $X_1 \subset 
\japanese[x]^{\sigma}L^2(\R^3)$ for each $\sigma > \frac12$.
Our assumptions on $V$ indicate that $VX_1$ is a subset of $L^1(\R^3)$.

Because $H$ is not a symmetric operator, its ``Jordan block'' structure at
each eigenvalue may not be limited to a direct sum of eigenfunctions. 
We would like to classify the deeper structure by recursively setting
$X_{k+1}$ to be the space of functions $\Psi$ such that $H\Psi \in X_k$, 
$k \ge 0$.  This heuristic is imprecise with regard to the acceptable
level of spatial growth/decay of $\Psi \in X_k$, so we introduce an inductive
set of assumptions in order to define $X_k$ more carefully.  The base case
assumption is
\begin{equation*}
X_1 \subset L^1(\R^3).
\end{equation*}
Assuming at each step that $X_k$ is a subspace of $L^1(\R^3)$, we are able to
define
\begin{equation*}  
X_{k+1} := \{\Psi \in L^3_{\rm weak}(\R^3): (-\Delta + V)\Psi \in X_k\}.
\end{equation*}

Now we are prepared to state the theorem.

\begin{theorem}
\label{thm:dispersive}
Let $V \in L^p(\R^3) \cap L^q(\R^3)$, $p < \frac32 < q$ be a complex-valued
potential.
Assume that $H = -\Delta + V$ has no resonances, and that the generalized
eigenspace at zero energy satisfies the the assumptions
\renewcommand{\theenumi}{(A\arabic{enumi})}
\renewcommand{\labelenumi}{\theenumi}
\begin{enumerate}
\item \label{A1}
$X_k \subset L^1(\R^3)$ for each $k \ge 1$.
\item \label{A2}
$X := \bigcup_k X_k$ is a finite-dimensional space.
\end{enumerate}
Then
\begin{equation} \label{eq:dispersive}
\big\|e^{-itH} (I - P_{pp}(H))f\big\|_{\infty} 
  \les |t|^{-\frac32}\norm[f][1]. 
\end{equation}
Under these conditions, the spectral multiplier $P_{pp}(H) := 
\sum_j P_{\lambda_j}(H)$ is a finite-rank projection.
\end{theorem}

The propagator $e^{-itH}$ is not a unitary operator unless $V$ is real-valued,
and indeed the evolution of complex eigenvalues and/or
generalized eigenfunctions must be unbounded at large times. 
Once these exceptional
vectors are projected away, then a global in time $L^2$ bound can be
recovered.  The following statement is adapted from Theorem~1 in~\cite{Go08},
which is based on the endpoint
inhomogeneous Strichartz estimate for the Laplacian~\cite{KeTa98}.

\begin{theorem} \label{thm:L2bound}
Under the same assumptions on Theorem~\ref{thm:dispersive},
\begin{equation} \label{eq:L2bound}
\big\|e^{-itH} (I - P_{pp}(H))f\big\|_2
  \les \norm[f][2].
\end{equation}
\end{theorem}

There is a slight difference in that the cited result assumes the absence
of generalized eigenfunctions at zero and elsewhere.  However the construction
in Sections~\ref{sec:decomp} and~\ref{sec:inverse} here can be readily
incorporated into the prevailing line of argument, making a separate proof
unnecessary.

The first dispersive estimate of the form~\eqref{eq:dispersive} was proved
in~\cite{JoSoSo91} for real potentials satisfying
a regularity hypothesis $\hat{V} \in L^1(\R^3)$ and pointwise decay bound
$|V(x)| \les \japanese[x]^{-7-\eps}$.  Successive improvements (\cite{Ya95},
~\cite{GoSc04}, and~\cite{Go06b} in chronological order)
relaxed the requirements
on $V$ down to the $L^p \cap L^q$ condition found in the statement of
Theorem~\ref{thm:dispersive}.  In all these results $H$ is assumed not to
possess an eigenvalue or resonance at zero.

The importance of the zero-energy spectral condition was already understood
in earlier work on Schr\"odinger operators in weighted
$L^2(\R^3)$.  Notably,
the asymptotic series expansion for $e^{-itH}$ developed in~\cite{JeKa79}
shows a leading term of order $|t|^{-1/2}$ (rather than the
desired $|t|^{-3/2}$) whenever $H$ has a resonance at zero.  Furthermore,
when $\lambda = 0$ is an eigenvalue this term typically
does not vanish even after projecting away from the eigenfunction(s)
with $I - P_0(H)$.  The same phenomenon is observed in~\cite{ErSc04}
and~\cite{Ya05} in the $L^1 \to L^\infty$ setting.
One general
interpretation of these findings is that threshold eigenvalues of $H$ cast a
significant ``shadow'' onto adjacent parts of the continuous spectrum.

Recent progress on the dynamical behavior of semilinear Schr\"odinger
equations has created a need for dispersive estimates associated to
non-selfadjoint Hamiltonians.  With a model nonlinearity that depends
on $|u(t,x)|^2$, the linearization around any nonzero standing wave
solution will feature coupled equations for the discrepancy $w(t,x)$ and its
complex conjugate $\overline{w(t,x)}$.  The resulting 2$\times$2 system of
Schr\"odinger equations has a matrix ``potential'' whose diagonal elements are
real and whose off-diagonal elements are skew-Hermitian.  In any analogy to
the scalar case these would correspond to the real and imaginary parts of 
our potential $V(x)$.

Stability of the nonlinear equation (or a
characterization of its stable manifold) depends largely on the spectrum of
this linearized operator, which is known for a limited number of examples.
Dispersive estimates then play an important role in controlling the
nonlinear interactions that appear in the course of a fixed-point argument.
A representative sample of results can be found in~\cite{Be08}, \cite{Cu01},
\cite{RoScSo05}, and~\cite{SoWe04}.

Even in these settings it is currently
necessary to assume that the Hamiltonian does not possess any eigenvalues
at the threshold or embedded in its continuous spectrum.  That is to say,
the existing dynamical analysis depends on this property which is not known
to hold in all cases.  The Schr\"odinger operator presented here has a simpler
structure, but we believe it serves as a test case to suggest that eigenvalues
may be admissible in all parts of the spectrum so long as the associated
eigenfunction possesses sufficient decay at infinity.

Returning to Theorem~\ref{thm:dispersive}, much of the work is concentrated in
choosing the correct set of projections in its statement.
The structure of $P_0(H)$ will be examined in detail in
Section~\ref{sec:decomp}, and is representative of all $P_\lambda(H)$. 
One characterization is that $P_{\lambda}(H)$ is
the projection whose range is the generalized eigenspace of $H$ 
over an eigenvalue $\lambda$, and whose adjoint has the generalized
eigenspace of $H^*$ over $\bar{\lambda}$ as its range.  For isolated 
eigenvalues $\lambda \not\in \R^+$
the analytic Fredholm theorem guarantees that
$P_\lambda(H)$ has finite rank.
These projections can also be
defined as an element of the analytic functional calculus of $H$ (as in
\cite{Ar02}, Chapter 1).
Specifically, $P_\lambda(H)$ corresponds to a function that is identically 1 in
a neighborhood of $\lambda$ and vanishes near the remaining spectrum of $H$.
The orthogonality relations $P_\lambda(H)P_\mu(H) = 0$ continue to hold for all
$\lambda \not= \mu$ even when $H$ is not self-adjoint.

The analytic functional calculus is less obviously applicable to eigenvalues
embedded within the essential spectrum $[0,\infty)$ or at its endpoint.
It is shown in~\cite{IoJe03} that Schr\"odinger operators with a complex
potential $V \in L^{3/2}(\R^3)$ possess no eigenvalues along the positive
real axis, however examples with an eigenvalue at $\lambda = 0$ are easily
constructed.

The proof of Theorem~\ref{thm:dispersive} is largely based on elementary
observations, however several detours are needed in order to gather the
relevant background.  The road-map proceeds as follows: 
Section~\ref{sec:reduction} describes a Fourier transform argument that
reduces the problem to an operator estimate involving certain
resolvents, whose basic properties are outlined in 
Section~\ref{sec:resolvents}.  Separate calculations are then required for
high, intermediate, and low energies, with the first two being borrowed
nearly intact from~\cite{Go06b}.  Some details are sketched out in
Section~\ref{sec:highmedenergy} for future reference. 
In Section~\ref{sec:decomp} we examine the structure of $X$ and its 
duality relationship to $\overline{X}$ in order to construct the projection
$P_0(H)$.  Section~\ref{sec:inverse} assembles the resolvent of $H$ in the
neighborhood of its threshold eigenvalue, and the concluding section
estimates its behavior once the poles associated to each (generalized)
eigenfunction are projected away.

\section{Reduction to Resolvent Estimates} \label{sec:reduction}

Let $H=-\Delta+V$ in $\R^3$ and define the resolvents
$R_0(z) :=(-\Delta-z)^{-1}$ and $R_V(z):= (H-z)^{-1}$.
For each $z \in \Compl \setminus \R^+$, the operator $R_0(z)$ 
is given by convolution with the kernel
\begin{equation*}
R_0(z)(x,y) = \frac{e^{i\sqrt{z}|x|}}{4\pi|x|},
\end{equation*}
where $\sqrt{z}$ is taken to have positive imaginary part.
While $R_V(z)$ is not translation--invariant and
does not possess an explicit representation of this form,
it can be expressed in terms of $R_0(z)$ via the identity
\begin{equation} \label{eq:ResIdent}
R_V(z) = (I + R_0(z)V)^{-1}R_0(z) = R_0(z)(I + VR_0(z))^{-1} 
\end{equation}
In the case where $z = \lambda \in \R^+$, the resolvent may be defined
as a limit of the form $R_0(\lambda \pm i0) := \lim_{\eps\downarrow 0} 
R_0(\lambda\pm i\eps)$.  The choice of sign determines which branch of the
square-root function is selected in the formula above, therefore the two
continuations do not agree with one another.  
In this paper we refer to resolvents along the positive real axis using
the following notation.
\begin{equation*}
R_0^\pm(\lambda) := R_0(\lambda\pm i0) \qquad
R_V^\pm(\lambda) := R_V(\lambda\pm i0)
\end{equation*}

Note that $R_0^-(\lambda)$ is the formal adjoint of $\Res(\lambda)$,
and more generally $R_V^\pm(\lambda)$ is the formal adjoint of
$R_{\overline{V}}^\mp(\lambda)$.
In fact these resolvents will be truly adjoint to one another,
since our analysis takes place in settings where they are 
bounded operators.  Determining the adjoint of $H$ is a more
delicate matter because of technicalities related to its domain,
but it can be done because $V\in L^{\frac32}(\R^3)$,
see Theorem X.19 in~\cite{ReSi2}.

If $V$ were real-valued and satisfied the $L^p$ condition in
Theorem~\ref{thm:dispersive} (making $H$ self-adjoint), then the Stone
formula for the absolutely continuous spectral measure of $H$ would
dictate that
\begin{equation*}
e^{-itH}f = \sum_j e^{-i\lambda_j t} \la f,\psi_j\ra\psi_j
+ \frac{1}{2\pi i}\int_0^\infty e^{-it\lambda}
[R_V^+(\lambda)-R_V^-(\lambda)]f \,d\lambda.
\end{equation*}
The summation takes place over a countable number of eigenvalues
$\lambda_j$ with associated eigenfunctions $\psi_j$.  A similar ansatz,
modified to correctly portray the evolution of generalized eigenfunctions
of a non-selfadjoint $H$, is
valid for complex $V$ as well.  Here the spectral decomposition
takes the form
\begin{equation} \label{eq:Stone}
\begin{aligned}
e^{-itH}f = \sum_{\lambda_j \not= 0}  e^{-itH}P_{\lambda_j}(H)f 
 &+ e^{-itH}P_0(H)f \\
&+ \frac{1}{2\pi i}\int_0^\infty e^{-it\lambda}
[R_V^+(\lambda)-R_V^-(\lambda)]f \,d\lambda.
\end{aligned}
\end{equation}
The range of each $P_{\lambda_j}(H)$ is a finite-dimensional subspace that is
invariant under the action of $H$.  The restriction of $H$ to this space
possesses a single eigenvalue $\lambda_j$, thus $e^{-itH}$ behaves here like
$e^{-i\lambda_j t}$ times a square matrix that grows polynomially in $t$.

In both cases the projection
$I - P_{pp}(H)$ removes eigenfunctions in such a way that the
initial sum vanishes.  The success of dispersive estimates is then governed
by the integral term, in particular by its regularity approaching the
endpoint at zero.
It is customary to view the right-hand integral, together with the discrete
$\lambda=0$ term, as a contour integral in the complex plane.
Making a change of variables $\lambda \mapsto \lambda^2$ opens up the
contour along the real axis, with the understanding that
\begin{equation*}
R_V^+(\lambda^2) = \lim_{\eps\downarrow 0}R_V((\lambda + i\eps)^2)
 = \lim_{\eps\downarrow 0}R_V(\lambda^2 + i\,{\rm sign}\,(\lambda)\eps)
\end{equation*}
for all $\lambda \in \R$.  Under this new notation, the integral term
(representing the absolutely continuous part of the Schr\"odinger evolution)
combines with the contribution of threshold eigenvalues to yield
\begin{align*}
\frac{1}{\pi i}&\int_{-\infty}^\infty e^{-it\lambda^2} \lambda
R_V^+(\lambda^2)f \,d\lambda \\
&=\ \frac{1}{\pi i}\int_{-\infty}^\infty e^{-it\lambda^2} \lambda
\Res(\lambda^2)(I + V\Res(\lambda^2))^{-1}f\, d\lambda.
\end{align*}
A formal integration by parts leads to the expression
\begin{equation*}
\frac{1}{2\pi t}\int_{-\infty}^\infty e^{-it\lambda^2}
\big(I + \Res(\lambda^2)V\big)^{-1}\frac{d}{d\lambda}\Big[\Res(\lambda^2)
\Big] \big(I+V\Res(\lambda^2)\big)^{-1}f\, d\lambda.
\end{equation*}
If we adopt the shorthand notation $T^+(\lambda) := 
(I + V\Res(\lambda^2))^{-1}$ and also take $T^-(\lambda) :=
(I + \overline{V}R_0^-(\lambda^2))^{-1}$,
Theorem~\ref{thm:dispersive} should be
equivalent to the estimate
\begin{equation} \label{eq:reduction}
\Big|\int_{-\infty}^\infty e^{-it\lambda^2} 
\Big\la \frac{d}{d\lambda}\Big[\Res(\lambda^2)\Big] 
T^+(\lambda) f,\, T^-(\lambda) g\Big\ra\, d\lambda \Big| \\
\les |t|^{-\frac12}
\norm[f][1] \norm[g][1]
\end{equation}
holding for all $f \in {\rm ran} (I-P_{pp}(H)) \subset L^1$ and 
$g \in {\rm ran}\,(I-P_{pp}(H)^*) \subset L^1$.  It is not necessary to
test every $g \in L^1$ because of the operator identity
\begin{equation*}
e^{-itH}(I-P_{pp}(H)) \  = \ 
(I-P_{pp}(H))e^{-itH}(I-P_{pp}(H))
\end{equation*}
which follows from $I-P_{pp}(H)$ being an idempotent function of $H$.

\begin{remark}
To conduct the formal steps properly, one should introduce a smooth cutoff
$\chi(\lambda/L)$ into the integrand, then take limits as $L \to \infty$.
The novel analysis in this paper takes place when $\lambda$ is confined to
a compact neighborhood of zero, where it is reasonable to assume that any such
cutoff function is identically 1.
\end{remark}

The integral in \eqref{eq:reduction} can be evaluated via Plancherel's
identity.  First note that 
$e^{-it\lambda^2}\frac{d}{d\lambda}[\Res(\lambda^2)]$ is a family of
convolution operators represented explicitly by the kernel
$(-4\pi i)^{-1}e^{-it\lambda^2 + i \lambda |x|}$. 
Taking $\rho$ as the variable dual to $\lambda$, its Fourier transform 
should be a family of convolution operators with kernel
\begin{equation*}
 \big({\txt \frac{d}{d\lambda}}[\Res(\lambda^2)]\big)\hat{\phantom{i}}(\rho)
{\scr (x)} = (16\pi i t)^{-\frac12} e^{i\frac{(\rho - |x|)^2}{4t}}.
\end{equation*}
In particular this is bounded by $|t|^{\frac12}$ for every value of $\rho$
and $|x|$.  Theorem~\ref{thm:dispersive} then follows from as appropriate
$L^1$ estimate on the Fourier transform of $T^\pm(\lambda)$.  This is the
approach taken in~\cite{Go06b}, whose principal technical result is stated
below.
\begin{theorem}
Let $V \in L^p(\R^3) \cap L^q(\R^3)$, $p < \frac32 < q$, be a real-valued
potential and suppose that $H = -\Delta + V$ does not have a resonance or
eigenvalue at zero.  Then
\begin{equation*}
\bignorm[(T^\pm)\hat{\,}(\rho)f][L^1_\rho L^1_x] \les \norm[f][1]
\end{equation*}
for all $f \in L^1(\R^3)$.
\end{theorem}

By contrast, if there exists an eigenvalue or resonance at zero, then
$T^\pm(\lambda)$ must have a corresponding pole.  One cannot expect the Fourier
transform of an unbounded family of operators to satisfy any kind of $L^1$
estimate.  The proof of Theorem~\ref{thm:dispersive} therefore hinges on a
modified $L^1$ bound which carefully avoids any poles.

\begin{theorem} \label{thm:goal}
Let $V \in L^p(\R^3) \cap L^q(\R^3)$, $p < \frac32 < q$, be a complex
potential.  Suppose that $H = -\Delta + V$ has no resonances along the 
interval $[0,\infty)$, and that the eigenspace at zero satisfies
assumptions~\ref{A1} and ~\ref{A2}.  Then
\begin{equation*}
\bignorm[(T^+)\hat{\,}(\rho)f][L^1_\rho L^1_x] \les \norm[f][1]
\end{equation*}
for all $f \in L^1(\R^3)$ orthogonal to $\X$ in the sense
that $\int f(x)\Psi(x)\,dx = 0$ for each $\Psi \in X$.  By taking complex
conjugates, the same estimate is true of $(T^-)\hat{\,}(\rho)$ applied to
all integrable functions $g \perp X$.
\end{theorem}

The class of admissible $f$ should be recognized as the range of $I - P_0(H)$.
Working backward, Theorem~\ref{thm:goal} implies that the
original integral term in~\eqref{eq:Stone}
is always
bounded in $L^\infty(\R^3)$ with norm controlled by
$|t|^{-\frac32}\norm[f][1]$.  The proof of Theorem~\ref{thm:dispersive}
follows by adding in, then projecting away,
the contributions of each discrete eigenvalue $\lambda_j$.

\section{Resolvent Estimates and Resonances} \label{sec:resolvents}
We catalog several properties of the free Schr\"odinger resolvent for
future use.  These are all immediate consequences of the convolution
representation 
\begin{equation} \label{eq:FreeRes}
\R_0^\pm(\lambda^2)(x) = \frac{e^{\pm i\lambda|x|}}{4\pi|x|}.
\end{equation}
It will be assumed at all times that $V \in L^p(\R^3) \cap L^q(\R^3)$,
$p < \frac32 < q$, and is measured according to the norm
\begin{equation*}
\norm[V][] := \max\big(\norm[V][p], \norm[V][q]\big).
\end{equation*}

\begin{proposition} \label{prop:continuity}
The convolution kernel of $\Res(\lambda^2)$ 
belongs to $L^3_{\rm weak}(\R^3)$ uniformly for all $\lambda \in \R$.
Moreover, given any $p < \frac32$, the convolution kernel of $\Res(\lambda^2)
- \Res(\mu^2)$ belongs to $L^{p'}(\R^3)$ with a norm uniformly bounded by
$|\lambda - \mu|^{1-\frac{3}{p'}}$.

Given any function $V \in L^p \cap L^q$, $p < \frac32 < q$,
the family of operators $I + V\Res(\lambda^2)$ 
are each bounded on $L^1(\R^3)$ and vary continuously with the parameter
$\lambda \in \R$.
\end{proposition}

\begin{proof}
The first two assertions are verified by direct inspection of the free
resolvent kernel in~\eqref{eq:FreeRes}.  They show that $\Res(\lambda^2)$ is
a bounded map from $L^1(\R^3)$ to $L^{p'}(\R^3) + L^{q'}(\R^3)$,
and provide a norm estimate for the difference $\Res(\lambda^2) - \Res(\mu^2)$.
Multiplication by $V$ then maps this space back to $L^1(\R^3)$ by H\"older's
inequality.  
\end{proof}

\begin{proposition} \label{prop:compactness}
Suppose $V \in L^p(\R^3) \cap L^q(\R^3)$, $p < \frac32 < q$.
For each $\lambda \in \R$, $V\Res(\lambda^2)$ is a compact operator from
$L^1(\R^3)$ to itself.
\end{proposition}

\begin{proof}
If $V$ were a smooth, bounded function with compact support, then 
$(-\Delta+1)V\Res(\lambda^2)$ would also map $L^1(\R^3)$ to itself thanks to
the identity $(-\Delta+1)\Res(\lambda^2) = I + (\lambda^2+1)\Res(\lambda^2)$.
In that case the target space of $V\Res(\lambda^2)$ consists of functions
in $W^{2,1}(\R^3)$ whose support is contained in ${\rm supp}\, V$.  This embeds
compactly into $L^1(\R^3)$ as desired.

The operator norm of $V\Res(\lambda^2)$ is always controlled by $\norm[V][]$.
So for general potentials it suffices to express $V$ as a norm limit of
smooth compactly supported functions.
\end{proof}

It is less immediate, but still well known that $I + V\Res(\lambda^2)$
becomes a small perturbation of the identity once $|\lambda|$ is sufficiently
large.

\begin{proposition} \label{prop:decay}
Suppose $V \in L^p(\R^3) \cap L^q(\R^3)$, $p < \frac32 < q$.  Then
\begin{equation*}
\lim_{|\lambda| \to \infty} \norm[(V\Res(\lambda^2))^2][1 \to 1] = 0.
\end{equation*}
\end{proposition}
\begin{proof}
There are numerous mapping estimates for the free resolvent which incorporate
some degree of decay as $\lambda \to \infty$.  To name one example,
\cite{KeRuSo87} shows that $\Res(\lambda^2)$ is a bounded operator
from $L^\frac43$ to $L^4$ with norm comparable to $|\lambda|^{-\frac12}$.
By interpolation with the elementary kernel-derived bounds above, one concludes
that $\Res(\lambda^2): L^{1+\eps} \to L^{3+\eps'}$ with a polynomial decay
rate determined by scaling.

Putting these pieces together, $V\Res(\lambda^2)$ maps $L^1$ to $L^{1+\eps}$
with unit norm, and it takes $L^{1+\eps}$ back to $L^1$ with an operator
norm controlled by $|\lambda|^{-\eps}$ for some $\eps > 0$.
\end{proof}

In order for its Fourier transform to satisfy an estimate in $L^1$, the
family of operators $(I + V\Res(\lambda^2))^{-1}$ should be uniformly bounded
in norm.  The above proposition suffices to show uniform boundedness for large
$\lambda$, where a Neumann series for the inverse will be convergent.     
The continuity assertion of Proposition~\ref{prop:continuity} then reduces the
problem to one of pointwise existence.

Since each operator $I + V\Res(\lambda^2)$ is a compact perturbation of the
identity, the Fredholm alternative stipulates that inverses must exist except
in the case where $(I + V\Res(\lambda^2))g = 0 $ has a solution $g \in L^1$.
It would follow that $\Psi = \Res(\lambda^2)g$, which belongs to
various spaces such as $L^3_{\rm weak}(\R^3)$ and $\japanese[x]^\sigma L^2$ for
all $\sigma > \frac12$, is a distributional solution of
$(-\Delta + V - \lambda^2)\Psi = 0$.  In the event that $\Psi \not\in L^2$,
we say that $H$ has a {\em resonance} at $\lambda^2$.

It is shown in~\cite{IoJe03} that $H$ can only possess true eigenvalues at
$\lambda^2 \in \R^+$ when $\lambda = 0$.  Furthermore, for real valued
potentials there is a self-adjointness argument~\cite{Ag75} to rule out the
possibility of resonances at any nonzero $\lambda$.  We have adopted
condition~\ref{A2} in Theorem~\ref{thm:dispersive} in order to avoid
embedded resonances even for complex potentials. With the invertability of
$I + V\Res(\lambda^2)$ now ensured for all $\lambda \in \R \setminus \{0\}$,
Propositions~\ref{prop:continuity} and~\ref{prop:decay} lead to the uniform
bound
\begin{equation} \label{eq:UnifInverse}
\sup_{|\lambda| > \lambda_1} \norm[(I+V\Res(\lambda^2))^{-1}][1\to 1]
< \infty
\end{equation}
for every $\lambda_1 > 0$.

Our attention now turns to the eigenvalues that may occur at $\lambda = 0$.
As before, the only obstruction to the invertability of $I + V\Res(0)$
is the existence of nonzero $g \in L^1(\R^3)$ such that
$(I + V\Res(0))g = 0$.  These solutions then generate our space $X_1$
via the relationship
\begin{equation} \label{eq:X_1}
X_1 = \{\Res(0)g: (I + V\Res(0))g = 0, g \in L^1 \}.
\end{equation}
One consequence of the Fredholm Alternative is that the space of admissible
functions $g$ must be finite-dimensional, so $X_1$ has finite dimension as
well.

\begin{proposition} \label{prop:bootstrap}
Let $V \in L^p(\R^3) \cap L^q(\R^3)$, $p < \frac32 < q$, and suppose that
Assumption~\ref{A1} holds.  Then $X_k \in L^1(\R^3) \cap L^\infty(\R^3)$
for each $k \ge 1$.
\end{proposition}

\begin{proof}
The kernel estimate in Proposition~\ref{prop:continuity} immediately
suggests that $X_1$ is contained in $L^3_{\rm weak}(\R^3)$, based on the
{\it a priori} information that $g \in L^1$.  The assumptions on $V$ then
indicate that $V\Psi \in L^r(\R^3)$ for all $1 \le r < 1 + \frac{2q-3}{q+3}$.
Each function $\Psi \in X_1$ satisfies the recurrence relation
\begin{equation*}
\Psi = -\Res(0)V\Psi = \Delta^{-1} V\Psi.
\end{equation*}
A bootstrapping argument, estimating the right-hand side
with Sobolev inequalities and the assumption that $V \in L^q$ for some
$q > \frac32$, proves that $\Psi \in L^\infty$ (and even $W^{2,q}(\R^3)$).
Meanwhile, $\Psi \in L^1$ by direct assertion of~\ref{A1}.

The same approach applies equally well to $X_k$.  By definition 
$\Psi \in X_k$ precisely if there is some element $\Psi_0 \in X_{k-1}$
for which
\begin{equation*}
\Psi = -\Res(0)(V\Psi-\Psi_0) = \Delta^{-1}(V\Psi - \Psi_0).
\end{equation*}
Assuming inductively that $\Psi_0$ belongs to every $L^p(\R^3)$ makes it
possible to place $\Psi \in L^\infty$ by bootstrapping.  The fact that
$X_k \subset L^1$ is again part of the statement~\ref{A1}.
\end{proof}

This leads to a desirable state of affairs: Orthogonal projection away from
the space of eigenfunctions $X_1$ (or generalized eigenfunctions $X$) acts as a
bounded operator on both $L^1(\R^3)$ and $L^\infty(\R^3)$.

\section{The High-Energy and Intermediate-Energy Cases}
\label{sec:highmedenergy}

Theorem~\ref{thm:goal} is in fact a compilation of three statements, since
distinct methods are used to analyze $T^\pm(\lambda)$ depending on the size of
$|\lambda|$.  Separating the different energy regimes is done by means of a 
smooth cutoff $\chi(\lambda)$ that is identically equal to 1 when 
$|\lambda| \le 1$ and vanishes for $|\lambda| \ge 2$. 
We are able to borrow the high-energy and intermediate-energy
results from~\cite{Go06b} with minimal modification; the precise
statements are formulated below.
 
\begin{theorem} \label{thm:highenergy}
Let $V\in L^p(\R^3)\cap L^q(\R^3)$, $p<\frac32 < q$ be a complex-valued
potential. There exists a number $\lambda_1 < \infty$ depending only on
$\norm[V][]$ so that the inequality
\begin{equation} \label{eq:high}
\int_{\R^3}\int_\R \big|
  \big[(1 - \chi(\,\cdot/\,\lambda_1))T^+(\,\cdot\,)\big]^\wedge(\rho)f(x)
 \big|\,d\rho\,dx \les \norm[f][1]
\end{equation}
holds for all $f \in L^1(\R^3)$.  
\end{theorem}

The proof is based on the Neumann series expansion $(I + V\Res(\lambda^2))^{-1}
= \sum_m (-V\Res(\lambda^2))^m$.  A finite number of terms at the beginning
are computed directly from the resolvent kernel~\eqref{eq:FreeRes}, and 
for small $V$ it is even possible to control the entire series in this manner
without any further restrictions or assumptions~\cite{RoSc04}.  Otherwise
the high-energy cutoff is needed so that the $m^{\rm th}$ term can be
bounded by $(\lambda_1^{-\eps}\norm[V][])^m$ (estimates such as 
Proposition~\ref{prop:decay} play a large role here), leading to a convergent
geometric series.

We refer the reader to~\cite{Go06b} for the details.  While only
the case of real potentials is considered, the estimates are based solely
on the size of $V$ in $L^p \cap L^q$ and extend to the complex case without
modification.  Within the interval $\lambda \in (\lambda_1, \infty)$, 
convergence of the Neumann series for $(I + V\Res(\lambda^2))^{-1}$ expressly
forbids the existence of embedded eigenvalues or resonances of $H$.

\begin{theorem} \label{thm:medenergy}
Suppose $V$ satisfies the conditions of Theorem~\ref{thm:dispersive},
and fix any $0 < \lambda_1 < \infty$.  Then
\begin{equation} \label{eq:medenergy} 
\int_{\R^3}\int_\R \big|
\big[(\chi(\,\cdot\,/\lambda_1)-\chi(\lambda_1\,\cdot\,)) T^+(\,\cdot\,)
  \big]^\wedge(\rho)f(x)\big|\,d\rho\,dx  \les \norm[f][1]
\end{equation}
holds for all $f \in L^1(\R^3)$. 
\end{theorem}

This statement is almost identical to Theorem~13 in~\cite{Go06b},
only with a small interval around zero energy removed so as to avoid any
possible eigenvalues there.  We reproduce an summary of the proof because
the more delicate low-energy analysis in Sections~\ref{sec:inverse} and
~\ref{sec:lowenergy} will be
adapted from these calculations.

Existence of $(I+V\Res(\lambda^2))^{-1}$ for each $\lambda$ in this range
is established via the Fredholm Alternative (with~\ref{A1} added by fiat as
a supporting assumption) rather than by power series methods.
This is a black-box procedure in that the outcome
cannot be expressed more directly in terms of $V$ and $\lambda$.  As a result,
the only known relationship between the various operators $T^+(\lambda)$
comes from the continuity statement in Proposition~\ref{prop:continuity}.

With this in mind, we introduce a secondary cutoff that localizes to much
smaller intervals of $\lambda$ where continuity arguments can be applied
successfully. Theorem~\ref{thm:medenergy} is proved by adding up a finite
number of local results.

Fix a ``benchmark'' energy $\lambda_0 \in \R$ and let 
$S_0 = (I+V\Res(\lambda_0^2))^{-1}$.  For all values of $\lambda$ 
sufficiently close to $\lambda_0$, we may regard $\Res(\lambda^2)$ as 
a perturbation of $\Res(\lambda_0^2)$ and treat the inverse of
$I + V\Res(\lambda^2)$ as a perturbation of $S_0$ according to the formula
\begin{equation*}
(I + V\Res(\lambda^2))^{-1} = 
(I + S_0VB_{\lambda_0}^+(\lambda^2))^{-1}S_0
\end{equation*}
where $B_{\lambda_0}^+(\lambda^2)$ represents the difference
$\Res(\lambda^2) - \Res(\lambda_0^2)$.

Let $\chi(\lambda)$ be a smooth cutoff that is identically equal to
1 when $|\lambda| \le 1$ and vanishes for all $|\lambda| \ge 2$.  Suppose
$\eta$ is any smooth function supported in the interval $[\lambda_0 - r,
\lambda_0 + r]$.  The operator inverse of $I + V\Res(\lambda^2)$ has a local
Neumann series expansion
\begin{equation} \label{eq:Neumann}
\eta(\lambda)(I + V\Res(\lambda^2))^{-1} = \eta(\lambda) \sum_{m=0}^\infty
 \big(-S_0VB_{\lambda_0}^+(\lambda^2)
   \chi\big({\txt \frac{\lambda-\lambda_0}{r}}\big)\big)^m S_0.
\end{equation}
Convergence of the sum is guaranteed by Proposition~\ref{prop:continuity}
provided $r$ is sufficiently small.  Note that the role of $S_0$ is limited
to its existence as a (fixed) bounded operator on $L^1(\R^3)$ whose norm is
controlled by~\eqref{eq:UnifInverse}.

In order to take the Fourier transform of~\eqref{eq:Neumann}, the key
step is to estimate the Fourier transform (with respect to $\lambda$) 
of $VB_{\lambda_0}^+(\lambda^2) \chi(r^{-1}(\lambda-\lambda_0))$. 
This is a family of integral operators with kernel
\begin{equation*}
K(x,y,\lambda) = V(x)\frac{e^{i(\lambda-\lambda_0)|x-y|}-1}{4\pi|x-y|}
e^{i\lambda_0|x-y|} \chi\big({\txt \frac{\lambda - \lambda_0}{r}}\big)
\end{equation*}
Its Fourier transform is also a family of integral operators with kernel
\begin{equation*}
\hat{K}(x,y,\rho) = V(x)e^{i\lambda_0(|x-y|-\rho)}
\frac{r\hat{\chi}(r(\rho-|x-y|)) - r\hat{\chi}(r\rho)}{4\pi |x-y|}
\end{equation*}
where $\rho$ is the Fourier variable dual to $\lambda$.  Using the Mean Value
inequality,
\begin{equation*}
\int_\R r\big|\hat{\chi}(r(\rho - |x-y|))- \hat{\chi}(r\rho) \big|\,d\rho
 \le \min (2\norm[\hat{\chi}][1], r|x-y|\norm[\hat{\chi}'][1])
\end{equation*}
therefore
\begin{equation*}
\int_\R |\hat{K}(x,y,\rho)|\,d\rho \les |V(x)|\min(r, |x-y|^{-1})
\end{equation*}
Then we can further integrate with respect to $x$ by taking $V$ in
$L^p \cap L^q$, with the result
\begin{equation}
\sup_{y\in\R^3}\int_{R^4} |\hat{K}(x,y,\rho)|\,dx d\rho \le C_{p,q} r^\eps 
\norm[V][]
\end{equation}
The exponent is $\eps = \min(\frac3p-2, 2-\frac3q) > 0$.  In other words, given
any function $f \in L^1$,
\begin{equation} \label{eq:VBhat}
\int_\R \norm[\hat{K}(\rho)f][1]\, d\rho \le C_{p,q} \norm[V][] r^\eps
\norm[f][1].
\end{equation}
Composing with the bounded operator $S_0$ has no effect except to change
the value of the constant $C_{p,q}$.  Multiplying the original family of
operators by $\eta(\lambda)$ causes a convolution by $\hat{\eta}(\rho)$ on
the Fourier side.  This increases the constants again by a factor of
$\norm[\hat{\eta}][1]$, by the standard $L^1$ convolution arguments.

The Fourier transform of the $m^{\rm th}$ term of the series is subject to the
estimate
\begin{equation*}
\int_\R \norm[\hat{K}_m(\rho)f][1]\, d\rho \le
\big(C_{p,q}\norm[V][]r^{\eps} \big)^m \norm[\hat{\eta}][1] \norm[f][1].
\end{equation*}
So long as $r$ is small enough that $C_{p,q} \norm[V][] r^{\eps} < \frac12$,
the Fourier transform of the Neumann series~\eqref{eq:Neumann} will converge 
as desired.  

We observe that convergence is not influenced by the specific
profile of $\eta$, so any smooth partition of unity for
$[\lambda_1^{-1}, \lambda_1]$ suffices to complete the proof of
Theorem~\ref{thm:medenergy} so long as each component is supported in an
interval of length less than $2r$.

\section{Decomposition and Duality  of $X$ and $\X$} \label{sec:decomp}

In the event that $I + V\Res(0)$ is an invertible operator on $L^1(\R^3)$,
it is possible to choose $\lambda = 0$ as a benchmark energy as well. 
The zero-energy cutoff $\chi(\lambda_1\,\cdot\,)$ in
Theorem~\ref{thm:medenergy} could then be removed without consequence.
This corresponds to the
(frequently invoked) assumption that zero is neither an eigenvalue nor a
resonance of $H$.  In our notation, an equivalent condition is that
$X = \{0\}$.

More generally, $X$ captures the algebraic structure of $H$ over its 
$\lambda = 0$ eigenvalue.  We would like to choose a ``Jordan basis'' for $X$
so that the action of $H$ has a simple representation.  In the process we will
recover numerous duality properties regarding the pairing of functions from
$X$ and $\X$.  The main (perhaps only) ingredient is the
orthogonality relation between the range of a bounded operator and the kernel
of its adjoint.  To avoid problems stemming from the domain of $H$ and $H^k$,
we will only consider the restrictions $H: X\to X$ and $H^*: \X
\to \X$, and formulate all other arguments in terms of resolvents.
The original definition of $X$ in Section~\ref{sec:intro}
was carried out in a similar spirit.

Recall that $X$ is a nested union of spaces $X_1 \subset X_2 \subset \ldots $.
It will be assumed at all times that properties~\ref{A1} and~\ref{A2} are
satisfied.

\begin{proposition}
For each $k \ge 1$, $H: X_{k+1} \to X_k$ and $X \cap H^{-1}X_k = X_{k+1}$. 
The kernel of $H$ (relative to $X$) is $X_1$.  The kernel of $H^k$ relative
to $X$ is therefore $X_k$.
\end{proposition}
\begin{proof}
A function $\psi_{k+1}$ belongs to $X_{k+1}$ provided that $\psi_{k+1} \in
L^3_{\rm weak}(\R^3)$ and $(I + \Res(0)V)\psi_{k+1} = \Res(0)\psi_k$ for
some $\psi_k \in X_k$.  Applying the Laplacian to both sides of this
equation yields $H\psi_{k+1} = \psi_k$.  Note that $V\psi_{k+1}$ and
$\psi_k$ are both in $L^1(\R^3)$, so the composition $-\Delta \Res(0) = I$
is valid when applied to them.  Thus functions $\psi \in X$ are
elements of $X_{k+1}$ precisely if $H\psi \in X_k$.  For similar reasons,
elements of $X_1$ have $H\psi = 0$ and these are the only such functions
in $X$.
\end{proof}

Assumption~\ref{A2} implies that in fact $X = X_k$ for some $k < \infty$.
Let $K \ge 0$ be the smallest such $k$ for which this holds.  We have a
finite dimensional vector space $X$ and a nilpotent linear transformation
$H:X \to X$.  There are numerous methods for constructing a basis for $X$
so that $H$ appears in Jordan normal form.  We describe one such construction
here and will further refine it during the discussion of duality later in this
section.

Fix an embedding of $X_1 / HX_2$ into $X_1$ and define this to be $X_{1,1}$.
For each $2 \le k \le K$, let $X_{k,k}$ be an embedding of $X_k/(X_{k-1} +
HX_{k+1})$ into $X_k$.  To complete the triangular array, define
$X_{j,k} = HX_{j+1,k}$ for each $1 \le j < k$.

\begin{proposition}
The following statements are valid for any set of spaces $X_{j,k}$ generated
by the above construction.
\begin{enumerate}
\item $HX \cap X_{k,k} = \{0\}$ for each $k$.
\item $H:X_{j,k} \to X_{j-1,k}$ is one-to-one for each $1 < j \le k$.
\item $X_j/ X_{j-1} \cong \bigoplus_{k=j}^{K} X_{j,k}$ for each 
$2 \le j \le K$, and $X_1 \cong \bigoplus_{k=1}^K X_{1,k}$.
\item $X \cong \bigoplus_{1 \le j \le k \le K}  X_{j,k}$.
\end{enumerate}
\end{proposition}

\begin{proof}
\begin{enumerate}
\item Any element of $X_{k,k} \cap HX$ must
belong to $HX_{k+1}$, which makes it an embedding of the identity element
in $X_k / X_{k-1} + HX_{k+1}$.
\item With the exception of the identity, every function $\psi_{k,k}
\in X_{k,k}$ fails to belong to $X_{k-1} = \ker_X H^{k-1}$.  Thus 
$H^{k-1}: X_{k,k} \to X_{1,k}$ is a one-to-one map.
\item This is easiest to prove inductively, starting with $j=K$.
By definition $X_{K+1} = X_K$,
so $HX_{K+1} \subset X_{K-1}$ and $X_{K,K} \cong X_K / X_{K-1}$.

Now suppose that $j \ge 2$ and 
$X_{j+1}/X_j \cong \bigoplus_{k=j+1}^K X_{j+1,k}$. 
We may write
\begin{align*}
X_j/X_{j-1} &\cong X_j / (X_{j-1} + HX_{j+1}) 
  \oplus HX_{j+1}/(X_{j-1} \cap HX_{j+1})\\
&\cong X_{j,j} \oplus HX_{j+1} / HX_j \\ 
&\cong X_{j,j} \oplus H\Big({\txt \bigoplus_{k=j+1}^K} X_{j+1,k}\Big)
\end{align*}
which proves the desired result because $H$ is a one-to-one map when
applied to $X_{j+1}/X_j$.  The $j=1$ case is proved the same way by adopting
the convention that $X_0 = \{0\}$.
\item Follows immediately from (3) and the fact that $X = X_K$.
\end{enumerate}
\end{proof}

The function spaces $\X_{j,k}$ form an identical decomposition of $\X$,
with $H^*$ taking the place of $H$.  We now turn our attention to the
relationships between $X$ and $\X$ in the setting of the $L^2$ inner product.
It will soon become apparent that $\X$ is a dual space to $X$ (and vice versa)
and the subdivision structure of $X_{j,k}$ is preserved by this identification.
The primary goal of this section is to construct an ``orthonormal'' basis,
as stated in the following lemma.

\begin{lemma} \label{lem:Xbasis}
There exists a basis for $X$ of the form
\begin{equation*}
\{\psi_{j,k}^{(\ell)}: 1 \le j \le k \le K,\, 1 \le \ell \le L_k\}
\end{equation*}
with the following properties.
\begin{enumerate}
\item $\{\psi_{j,k}^{(\ell)}: 1 \le \ell \le L_k\}$ is a basis for $X_{j,k}$.
\item $H\psi_{j,k}^{(\ell)} = \psi_{j-1,k}^{(\ell)}$ for each triple
$(j,k,\ell)$ with $j > 1$, and $H\psi_{1,k}^{(\ell)} = 0$.
\item $\big\la \psi_{j_1,k_1}^{(\ell_1)}, \bar{\psi}_{j_2,k_2}^{(\ell_2)}
\big\ra =
\begin{cases}
 1,\ {\rm if}\ (k_1,\ell_1) = (k_2,\ell_2)\ {\rm and}\ 
j_2 = k+1-j_1.\\
 0,\ {\rm otherwise.}
\end{cases}$
\end{enumerate}
\end{lemma}

The third statement asserts a duality relationship between $X$ and $\X$, and at
the level of subspaces it identifies $\X_{k+1-j,k}$ with the dual space
of $X_{j,k}$.
Using the basis produced by Lemma~\ref{lem:Xbasis}, we define the spectral
projection of $H$ onto the eigenspace at zero to be
\begin{equation}
P_0(H) f = \sum_{\substack{1 \le j \le k \le K\\1 \le \ell \le L_k}}
  \la f, \bar{\psi}_{k+1-j,k}^{(\ell)}\ra \psi_{j,k}^{(\ell)}.
\end{equation}
This operator is bounded as a map from $L^p(\R^3)$ to itself for every
$1 \le p \le \infty$, thanks to Proposition~\ref{prop:bootstrap}. 
We list two additional
properties of $P_0(H)$ that follow readily from its definition.

\begin{proposition}
$P_0(H)$ is the unique finite-rank projection whose range is given by $X$
and whose adjoint has range $\X$, therefore it does not depend on the
particular choice of basis $\psi_{j,k}^{(\ell)}$ employed during its
construction.  The operators $H$ and
$P_0(H)$ commute with one another when applied to any function in the domain
of $H$.
\end{proposition}

The same constructions can be carried out at any eigenvalue $\lambda_j$
of $H$, with $H - \lambda_j$ replacing $H$ in the statement of
Lemma~\ref{lem:Xbasis}.  Some arguments and definitions are actually made
simpler by the fact that $(-\Delta - \lambda)^{-1}$ is always 
an $L^2$-bounded operator for $\lambda \not\in [0,\infty)$.
To give one example, Condition~\ref{A2} is automatically satisfied for each
nonzero eigenvalue as a result of the Analytic Fredholm Theorem
(e.g.\ Theorem~VI.14 in~\cite{ReSi1}).  Condition~\ref{A1} is also
satisfied because the extended collection of free resolvent estimates allows
for much stronger bootstrapping arguments than were used in 
Proposition~\ref{prop:bootstrap}.

The resulting operators $P_{\lambda_j}(H)$ are precisely the pure-point
spectral restrictions indicated in Theorem~\ref{thm:dispersive}.

Returning to the eigenspace at zero,
we define a limited projection that deals only with $\X_1$ (the
kernel of $H^*$) and its dual, $X_{\rm diag} :=\bigoplus_{k} X_{k,k}$.
\begin{equation}
\tilde{P}_0 f = \sum_{\substack{ 1 \le k \le K\\ 1 \le \ell \le L_k}}
 \la f, \bar{\psi}_{1,k}^{(\ell)}\ra \psi_{k,k}^{(\ell)}
\end{equation}
The range of $\tilde{P}_0$ is $X_{\rm diag}$, which is a (non-canonical)
embedding of $X/HX$ into $X$.  In contrast, The range of $\tilde{P}_0^*$ is
necessarily $\X_1$ regardless of the choice of basis $\psi_{j,k}^{(\ell)}$.
The projection $\tilde{Q}_0 := I - \tilde{P}_0$ is of equal interest.
Its range consists of functions orthogonal to $\X_1$.

The proof of Lemma~\ref{lem:Xbasis} occupies the remainder of this section.
It has been subdivided into four smaller steps that amount to a series of
exercises in linear algebra and functional analysis.  

\begin{remark}
For our purposes it would suffice to construct a basis 
$\psi_{j,k}^{(\ell)}$ for $X$ and a dual basis $\vp_{j,k}^{(\ell)}$ for $\X$,
each satisfying the first two statements of Lemma~\ref{lem:Xbasis}
(or their formulation with $\X$ and
$H^*$, as appropriate).  There is no compelling reason
other than convenience to demand that $\vp_{j,k}^{(\ell)} = 
\bar{\psi}_{j,k}^{(\ell)}$, as we do.  It is therefore unwise to
draw any correspondence between $\psi \in X$ and $\bar{\psi} \in \X$ in the
remaining discussion except where such a connection is explicitly mentioned.
\end{remark}

\begin{proposition} \label{prop:orthogonality}
for $i = 1,2$, choose a pair of numbers $1 \le j_i \le k_i \le K$ such that
$j_1 + j_2 \le \max(k_1, k_2)$.  Then the spaces $X_{j_1,k_1}$ and 
$\X_{j_2,k_2}$ are mutually orthogonal.
\end{proposition}
\begin{proof}
Consider the case where $j_1 + j_2 \le k_1$.  Any function 
$\psi \in X_{j_1,k_1}$ can be expressed as $H^{k_1-j_1} 
\tilde{\psi}$ for some $\psi' \in X_{k_1,k_1}$.
It follows that for every $\vp  \in \X_{j_2,k_2}$,
\begin{equation*}
\la \psi , \vp_{j_2,k_2}\ra =
\la \psi', (H^*)^{k_1-j_1}\vp \ra = \la \tilde{\psi}, 0 \ra.
\end{equation*}
The proof is identical for the case $j_1 + j_2 \le k_2$, with the roles
of $\psi$ and $\vp$ reversed.
\end{proof}

\begin{lemma} \label{lem:X_jkduality}
For each $1 \le j \le k \le K$, the space $\X_{k+1-j,k}$ is naturally
identified with the dual of $X_{j,k}$ via the ordinary $L^2$ pairing
\begin{equation*}
\la \psi, \vp\ra = \int_{\R^3} \psi(x) \bar{\vp}(x)\,dx \qquad
\psi \in X_{j,k}, \vp \in \X_{k+1-j,k}
\end{equation*}
This identification is independent of the particular realization of $X_{j,k}$.
\end{lemma}

\begin{proof}
It suffices to consider only the cases where $j=k$.  Whenever $j <k$
we have the identity
\begin{equation} \label{eq:Htransfer}
\la \psi, \vp \ra = \la \psi', (H^*)^{k-j}\vp \ra
\end{equation}
where $\psi' \in X_{k,k}$ is the unique solution of $H^{k-j}\psi' = \psi$.
Recall that $(H*)^{k-j}$ is an invertible map between $\X_{k+1-j,k}$ and
$\X_{1,k}$.

To show that pairings between $X_{k,k}$ and $\X_{1,k}$ 
are independent of their respective realizations, recall that $X_{k,k}$ is an
embedding of the quotient space $X_k / (X_{k-1} +
HX_{k+1})$.  Meanwhile, $\X_{1,k}$ is an embedding of the quotient
$(H^*)^{k-1}\X_k/ (H^*)^k \X_{k+1}$.  Consider the resulting pairing of cosets
\begin{align*}
&\la \psi + X_{k-1} + HX_{k+1}, \vp + (H^*)^k \X_{k+1} \ra \\
&= \la \psi + {\txt \bigoplus\limits_{j < k}} X_{j,k} + 
  {\txt \bigoplus\limits_{k' > k}} X_{k,k'}, \, 
 \vp + {\txt \bigoplus\limits_{k' > k}} \X_{1,k'} \ra \\
&= \la \psi, \vp \ra
\end{align*}
Every one of the possible cross-terms, including those involving
$\psi \in X_{k,k}$ and $\vp \in \X_{1,k}$, vanishes by
Proposition~\ref{prop:orthogonality}.

Duality of $X_{k,k}$ and $\X_{1,k}$ is proved by induction starting with the
base case $k=1$.  We may assume that $X_{1,1}$ is nontrivial.  Since it is
finite dimensional, it suffices to check that no nonzero
element of $X_{1,1}$ is orthogonal to the entirety of $\X_{1,1}$.

Suppose that $\psi \in X_{1,1}$ were orthogonal to $\X_{1,1}$.
Then $\psi$ is orthogonal as well to all of $\X_1 = \bigoplus_k \X_{1,k}$
by Proposition~\ref{prop:orthogonality}.  This suggests that, since
$\psi \perp \ker H^*$, it should belong to the range of $H$.

A more precise argument is:
Any function in $\vp \in \X_1$ satisfies the identity
$\vp= -\R_0^-(0)\overline{V} \vp$.
Thus $\Res(0)\psi \in L^\infty$ must be
orthogonal to the space $\overline{V}\X_1 = \ker (I+\overline{V}R_0^-(0))$,
which is a Fredholm operator on $L^1(\R^3)$.
Then $\Res(0)\psi$ belongs to the range of its adjoint, $I + \Res(0)V$.
By definition that places $\Res(0)\psi \in (I+\Res(0)V)X_2$ and
$\psi \in H X_2$.  The only valid possibility is to have
$\psi = 0 \in X_{1,1}$ because it is an embedding of the identity element
of $X_1 / X_1 \cap HX_2$.

Now suppose that duality between $X_{j,j}$ and $\X_{1,j}$ 
has been established for each $1 \le j < k$, and there exists a function
$\psi \in X_{k,k}$ that is orthogonal to $\X_{1,k}$.  Then $\psi$ is
also orthogonal to $\bigoplus_{k' \ge k} \X_{1,k'}$ and pairing with $\psi$
induces a linear functional on each of $\X_{1,j}$, $1 \le j < k$.
Using the inductive hypothesis, there exists $\psi_{k-1} \in X_{k-1, k-1}$
so that $\psi - \psi_{k-1} \perp \bigoplus_{k' \ge k-1} \X_{1,k'}$.

This construction can be continued until there is a collection of functions
$\psi_j \in X_{j,j}$, $1 \le j < k$, so that $\psi' =\psi - \sum_j \psi_j \perp
\X_1$.  Initially we have $\psi' \in X_k$, but by following the previous
argument (orthogonality to $\ker H^*$) we arrive at the stronger conclusion
$\psi' \in HX_{k+1}$.  This implies $\psi$ itself belongs to
$HX_{k+1} + X_{k-1}$, in which case $\psi = 0$.
\end{proof}

It is already a consequence of Lemma~\ref{lem:X_jkduality} and
Proposition~\ref{prop:orthogonality} that given a nonzero $\psi \in X$,
the pairings $\la\psi, \vp\ra$ cannot vanish for all $\vp \in \X$.  In that
sense $X$ and $\X$ are mutually identified with each other's dual space.
The last bit of work is to generate a decomposition of $X$ into
$\bigoplus_{j,k} X_{j,k}$ that emphasizes the
duality between $X_{j,k}$ and $\X_{k+1-j,k}$.  Once this is done the desired
basis $\psi_{j,k}^{(\ell)}$ can be formed using a variant of the
Gram-Schmidt process.

\begin{lemma} \label{lem:Xduality}
There exists a particular realization of each $X_{k,k}$ so that
the dual space to $X = \bigoplus_{j,k} X_{j,k}$ is naturally identified with
$\bigoplus_{j,k} \X_{k+1-j,k}$ in the sense that $\X_{k+1-j,k}$ is dual to
$X_{j,k}$ and
\begin{equation}
\la \sum_{j,k} \psi_{j,k}, \sum_{j,k}\vp_{j,k}\ra
= \sum_{j,k} \la \psi_{j,k}, \vp_{k+1-j,k}\ra
\end{equation}
for any pairing with $\psi_{j,k} \in X_{j,k}$ and $\vp_{j,k} \in \X_{j,k}$.
\end{lemma}

\begin{proof}
Lemma~\ref{lem:X_jkduality} showed the duality of $\X_{k+1-j,k}$ and $X_{j,k}$.
The remaining task is to construct spaces $X_{k,k}$ so that
$X_{j_1,k_1} \perp \X_{j_2,k_2}$ in every case where $k_1 \ne k_2$ or
$j_1 + j_2 \ne k_2 + 1$.  Thanks to~\eqref{eq:Htransfer} it will suffice to
make $X_{k,k}$ orthogonal to each $\X_{j_2,k_2}$, $(j_2,k_2) \not= (1,k)$.  

We first tackle the cases where $k_2 < k$.  For convenience the construction 
is presented here with an induction argument over $k$, though in fact
the steps need not be taken in strict order.  When $k=1$ it is vacuously
true that $X_{1,1} \perp \X_{j_2,k_2}$ with $k_2 < 1$. 

Assume that $X_{j,j}$, $1 \le j < k$, have all been chosen to be orthogonal
to $\X_{j_2,k_2}$, $k_2 < j$.  Given an initial embedding
$\tilde{X}_{k,k} \subset X_k$ of $X_k / (X_{k-1} + HX_{k+1})$,
there is a collection of linear maps
\begin{equation*}
T_{1,j}: \psi \mapsto \la \psi, \,\cdot\,\ra
\end{equation*}
from $\tilde{X}_{k,k}$ to the dual space of $\X_{1,j}$, $j < k$. 
These maps can also
be presented in the form $\psi \in X_{k,k} \mapsto \psi_{1,j} \in X_{j,j}$
by duality of $X_{j,j}$ and $\X_{1,j}$.  Then by
Proposition~\ref{prop:orthogonality} the realization
\begin{equation*}
X_{k,k} = \Big\{\psi - \sum_{j=0}^{k=1} \psi_{1,j}: 
 \psi \in \tilde{X}_{k,k}\Big\}
\end{equation*}
is an embedding of $X_k / (X_{k-1} + HX_{k+1})$ into $X_k$ that is orthogonal
to each of the fixed spaces $\X_{1,j}$, $j < k$.

Now repeat the process with the additional base assumption that
$\tilde{X}_{k,k} \perp \X_{1,j}$ whenever $j < k$.  There is a new collection
of linear maps
\begin{equation*}
T_{2,j}: \psi \mapsto \la \psi, \,\cdot\,\ra
\end{equation*}
from $\tilde{X}_{k,k}$ to the dual space of $\X_{2,j}$, $j < k$.
These maps associate a function $\psi_{2,j} \in X_{j-1,j}$ to each
$\psi \in X_{k,k}$ by Lemma~\ref{lem:X_jkduality}.
Proposition~\ref{prop:orthogonality} then implies that the realization
\begin{equation*}
X_{k,k} = \Big\{\psi - \sum_{j=0}^{k=1} \psi_j: \psi \in \tilde{X}_{k,k}\Big\}
\end{equation*}
is an embedding of $X_k / (X_{k-1} + HX_{k+1})$ into $X_k$ that is orthogonal
to each $\X_{j_2,k_2}$, $j_2 \le 2$, $k_2 < k$.  After $k-1$ iterations
of this procedure, the resulting realization of $X_{k,k}$ will be orthogonal
to every $\X_{j_2,k_2}$ with $k_2 < k$.

The cases $k_2 > k$ follow immediately by taking complex conjugates.
Observe that for any pair $\psi \in X_{k,k}$, $\vp \in \X_{j_2,k_2}$,
\begin{equation*}
\la \psi, \vp\ra = \la H^{k_2-j_2}\psi, \vp'\ra 
= \overline{\la \bar{\vp'}, (H^*)^{k_2-j_2}\bar{\psi} \ra}
\end{equation*}
where $\vp' \in \X_{k_2,k_2}$ is the unique solution to $(H^*)^{k_2-j_2}\vp'
= \vp$.  The right side of this identity pairs $\bar{\vp'} \in X_{k_2,k_2}$
with a function in $\X_{k+j_2-k_2,k}$.  The value must in fact be zero
by construction because $k < k_2$.

Now we need to set $X_{k,k}$ orthogonal to $\X_{j_2,k}$ for all $j_2 > 1$.
The danger here is that modifying the initial embedding $\tilde{X}_{k,k}$ by
an element of $X_{j,k}$ changes the contents of $\X_{j_2,k}$ as well. 
As before, there is a linear map
$T_2: \psi \mapsto \la \psi,\,\cdot\,\ra$
from $\tilde{X}_{k,k}$ to the dual of $\X_{2,k}$, which may then be
identified with $X_{k-1,k}$.  The correct initial adjustment is to set
\begin{equation*}
X_{k,k} = \Big\{\psi - {\txt \frac12}T_2\psi: \psi \in \tilde{X}_{k,k}\Big\}
\end{equation*}
so that the pairing of $X_{k,k}$ with $\X_{2,k}$ looks like
\begin{align*}
\la \psi - {\txt \frac12}&T_2\psi, \,
\vp - {\txt \frac12} \overline{H^{k-2} T_2 (H^{-(k-2)}\bar{\vp})}\ra \\
& 
\begin{aligned}
= \la \psi,\vp\ra -{\txt \frac12}\la T_2\psi,\vp \ra 
 - {\txt \frac12}\la T_2 (&H^{-(k-2)}\bar{\vp}), (H^*)^{k-2}\bar{\psi}\ra \\
&+ {\txt \frac14} \la H^{k-2} T_2 (H^{-(k-2)}\bar{\vp}), \overline{T_2\psi}\ra
\end{aligned}
\\  
&= \la \psi,\vp\ra - {\txt \frac12}\la \psi, \vp\ra 
 -{\txt \frac12} \la H^{-(k-2)}\bar{\vp}, (H^*)^{k-2}\bar{\psi}\ra + 0
\\
&=0
\end{align*}
The last cross-term is always zero by Proposition~\ref{prop:orthogonality}
since it pairs an element of $X_{1,k}$ with an element of $\X_{k-1,k}$.

If we start with the assumption that $\tilde{X}_{k,k} \perp \X_{j,k}$ for all 
$2 \le j  < j_2$, define a map $T_{j_2}: X_{k,k} \to X_{k+1-j_2,k}$
so that $\la T_{j_2}\psi, \vp\ra = \la \psi, \vp\ra$ for all 
$\vp \in \X_{j_2,k}$.  Then the new embedding of $X_{k,k}$ where 
each $\psi\in\tilde{X}_{k,k}$ is replaced with $\psi-\frac12 T_{j_2}\psi$
will be orthogonal to $\X_{j_2,k}$ by following the calculation above.
The pairings of $X_{k,k}$ with any $\X_{j,k}$, $j < j_2$, are unaffected
by the change because each of the cross-terms vanishes under
Proposition~\ref{prop:orthogonality}.
\end{proof}

Any basis of $X_{k,k}$ naturally induces bases for each space $X_{j,k}$ by
taking its image under $H^{k-j}$.  Bases for $\X_{j,k}$, which represent
the dual of $X_{k+1-j,k}$, are obtained through complex conjugation.  
We verify that a Gram-Schmidt process exists for creating a
basis of $X_{k,k}$ that induces its own dual basis in $\X_{1,k}$.

\begin{proposition}
For each $k = 1, 2, \ldots, K$ there exists a basis of functions
$\psi^{(\ell)} \in X_{k,k}$ such that
\begin{equation*}
\la \psi^{(\ell)}, \overline{H^{k-1}\psi^{(m)}}\ra =
\la H^{k-1}\psi^{(\ell)}, \bar{\psi}^{(m)} \ra
 = \delta_{\ell m}
\end{equation*}
\end{proposition}

\begin{proof}
Start by choosing any nonzero $\psi \in X_{k,k}$.  
Lemma~\ref{lem:X_jkduality} guarantees the existence of
$\vp \in \X_{k,k}$ so that $\la H^{k-1}\psi, \vp \ra = A \not= 0$.  Taking
linear combinations of $\psi$ and $\bar{\vp}$, we see that 
the value of $\la H^{k-1}(z \psi + \bar{\vp}), \bar{z}\bar{\psi} + \vp\ra$
is a nonconstant linear or quadratic function of the complex variable $z$.
Choose $\psi^{(1)}$ to be an element of the form $z_0\psi + \bar{\vp}$
such that $\la H^{k-1}\psi^{(1)}, \bar{\psi}^{(1)}\ra = 1$.

Now suppose $\psi^{(\ell)}$, $1 \le \ell < L$, have been chosen so that
$\la H^{k-1} \psi^{(\ell)}, \bar{\psi}^{(m)} \ra = \delta_{\ell m}$.
Given a linearly independent element $\tilde{\psi}$, let
\begin{equation*}
\psi = \tilde{\psi} - \sum_{\ell = 1}^{L-1} \la H^{k-1}\tilde{\psi},
\bar{\psi}^{(\ell)}\ra \psi^{(\ell)}.
\end{equation*}
All pairings between $H^{k-1}\psi$ and $\bar{\psi}^{(\ell)}$ vanish by
construction.
Once again there exists $\tilde{\vp} \in \X_{k,k}$ so that
$\la H^{k-1}\psi, \tilde{\vp}\ra = A \not= 0$.  Moreover, it does not change
the value if $\tilde{\vp}$ is replaced by
\begin{equation*}
\vp = \tilde{\vp} - \sum_{\ell=1}^{L-1} \overline{\la H^{k-1}\psi^{(\ell)},
\tilde{\vp}\ra} \bar{\psi}^{(\ell)},
\end{equation*}
the difference being that $\la H^{k-1}\bar{\vp}, \bar{\psi}^{(\ell)}\ra = 0$
as well.
Then it is possible to choose $\psi^{(L)} = z_0\psi + \bar{\vp}$ in the same
manner as before to continue building out the ``self-dual'' basis.
\end{proof}

\section{Construction of $(I+V\Res(\lambda^2))^{-1}$.} \label{sec:inverse}

Even though $I + V\Res(0)$ may not be an invertible operator on $L^1(\R^3)$,
it is still possible to determine the inverse of $I+V\Res(\lambda^2)$
for nearby $\lambda$ by perturbation.  The first step is to use the techniques
from Section~\ref{sec:highmedenergy} to construct an approximate inverse over
a large subspace of $L^1(\R^3)$.

As a compact perturbation of the identity, $I + V\Res(0)$ has a finite
dimensional kernel (namely $VX_1$) and its range is a closed subspace of
$L^1(\R^3)$ of equal codimension.  By the Open Mapping theorem, the exists
a continuous linear map 
\begin{equation*}
S_0: {\rm range}(I+V\Res(0)) \to L^1(\R^3) / \ker(I+V\Res(0))
\end{equation*}
 with the property that
\begin{equation*}
(I + V\Res(0))S_0 = I_{{\rm range}(I +V\Res(0))}.
\end{equation*}
Both the domain and range of $S_0$ can be expressed in terms of the
spaces $X$ and $\X$.  The range of $I+V\Res(0)$ in $L^1$
consists of functions that
vanish when paired with any element of $\X_1 =\ker (I + R_0^-(0)\overline{V})$.
Meanwhile, since $VX_1 = \ker (I+V\Res(0))$ is finite dimensional there exist
subspaces of $L^1(\R^3)$ isomorphic to the quotient $L^1 / VX_1$.  We will
concentrate on one such embedding that is canonical in the context of this
discussion.

Observe that $R_0^-(0)\X_{\rm diag} = \bigoplus_{k} R_0^-(0)\X_{k,k}$
is dual to $VX_1$ in that
\begin{equation*}
\la V\psi, R_0^-(0)\vp \ra = \la \Res(0)V\psi, \vp\ra = -\la\psi, \vp\ra
\end{equation*}
for any $\psi \in X_{1,k}$, $\vp \in \X_{k',k'}$.  Given a function
$g \in L^1(\R^3)$ there is a unique choice of $\psi_k \in X_{1,k}$ so that
$\la g, R_0^-(0)\vp_k\ra = \la V\psi_k, R_0^-(0)\vp_k\ra$ for all
$\vp_k \in \X_{k,k}$.
Then the mapping $g \mapsto g - \sum_k V\psi_k$ has $VX_1$ as its kernel.
Its range is an embedding of $L^1/VX_1$ inside of $L^1(\R^3)$ consisting of
functions orthogonal to $R_0^-(0)\X_{\rm diag}$.

In summary, there exists a linear map $S_0: \X_1^\perp \to
(R_0^-(0)\X_{\rm diag})^\perp$ that is a one-sided inverse to
$I + V\Res(0)$ over the domain $\X_1^\perp$.  Both the domain and range of
$S_0$ are understood to be subspaces of $L^1(\R^3)$.  We express this 
property in the slightly redundant form
\begin{equation*}
\tilde{Q}_0 (I + V\Res(0))S_0 = {\rm Identity\ of\ }
\X_1^\perp \subset L^1(\R^3).
\end{equation*}
The next step is to find a perturbation of $S_0$ that will serve as the
inverse of $\tilde{Q}_0(I + V\Res(\lambda^2))$, taking the latter as a map from
$(R_0^-(0)\X_{\rm diag})^\perp$ to $\X_1^\perp$.  If $\lambda$ is sufficiently
small this can be done via a convergent Neumann series.  Start with the
decomposition
\begin{equation*}
\tilde{Q}_0 (I+V\Res(\lambda^2)) = S_0^{-1} +\tilde{Q}_0 VB_0^+(\lambda^2).
\end{equation*}
From this it follows that
\begin{equation} \label{eq:Slambda}
S(\lambda) = \sum_{m=0}^\infty (-1)^m
  \big(S_0 \tilde{Q}_0 VB_0^+(\lambda^2)\big)^m S_0
\end{equation}
will be the desired inverse of $\tilde{Q}_0(I+V\Res(\lambda^2))$.
Proposition~\ref{prop:continuity} asserts that $VB_0^+(\lambda^2)$ vanishes
as an operator on $L^1(\R^3)$ when $\lambda$ approaches 0, and the same is true
after composing with $\tilde{Q}_0$ and $S_0$.  Thus the series 
in~\eqref{eq:Slambda} converges absolutely for small values of $\lambda$.

Of course $S(\lambda)$ is not generally a true one-sided inverse of
$I + V\Res(\lambda^2)$ over the domain $\X_1^\perp$.  The difference
between $(I+V\Res(\lambda^2))S(\lambda)f$ and $f$ itself will be a linear
combination of the $\psi_{k,k}^{(\ell)}$, with coefficients determined by
its pairings with the corresponding $\bar{\psi}_{1,k}^{(\ell)} \in \X_1$.
If we can identify solutions to 
\begin{equation} \label{eq:inverse}
(I + V\Res(\lambda^2))\Psi_{k}^{(\ell)} = \psi_{k,k}^{(\ell)}
\end{equation}
then the construction of $(I + V\Res(\lambda^2))^{-1}$ will be complete.

In fact there is an explicit formula for each $\Psi_{k}^{(\ell)}$ based on the
defined relationships between the $\psi_{j,k}^{(\ell)}$.
\begin{proposition}
For each $2\le j\le k \le K$ and each $1 \le \ell \le L_k$, there is the
identity
\begin{equation}
(I + \Res(\lambda^2)V)\psi_{j,k}^{(\ell)} =
 \Res(\lambda^2)\big(\psi_{j-1,k}^{(\ell)} - \lambda^2\psi_{j,k}^{(\ell)}\big)
\end{equation}
and $(I + \Res(\lambda^2))\psi_{1,k}^{(\ell)}
= -\lambda^2\Res(\lambda^2)\psi_{1,k}^{(\ell)}$ in the case $j=1$.

The analogous identities for $(I + R_0^-(\lambda^2)\overline{V})
\bar{\psi}_{j,k}^{(\ell)}$ are also valid.
\end{proposition}
\begin{proof}
This is a direct consequence of the operator identity 
\begin{equation} \label{eq:ResIdent2}
\Res(\lambda^2) = (I + \lambda^2\Res(\lambda^2))\Res(0).
\end{equation}
As a consequence, $I + \Res(\lambda^2)V$ can be rewritten as
\begin{equation*}
(I + \lambda^2\Res(\lambda^2))(I + \Res(0)V) - \lambda^2\Res(\lambda^2).
\end{equation*}
By definition $(I + \Res(0)V)\psi_{j,k}^{(\ell)} 
= \Res(0)\psi_{j-1,k}^{(\ell)}$.  The proposition is then proved with one
more application of~\eqref{eq:ResIdent2}.
\end{proof}

\begin{corollary} \label{cor:telescope}
For each $1 \le k \le K$ and $1 \le \ell \le L_k$ there is the identity
\begin{equation*}
(I + \Res(\lambda^2)V)\sum_{j=1}^k \lambda^{2(j-1)}\psi_{j,k}^{(\ell)}
= - \lambda^{2k}\Res(\lambda^2)\psi_{k,k}^{(\ell)}.
\end{equation*}
\end{corollary}
Once again the ``adjoint'' statement relating $R_0^-(\lambda^2)$, 
$\overline{V}$, and $\bar{\psi}_{j,k}^{(\ell)}$ holds as well.
We state the solution formula for~\eqref{eq:inverse} as a final corollary.

\begin{corollary}
For each $1 \le k \le K$ and $1 \le \ell \le L_k$,
\begin{equation} \label{eq:exactinverse}
(I + V\Res(\lambda^2))\Big(\psi_{k,k}^{(\ell)} + \sum_{j=1}^k
 \lambda^{-2(k+1-j)}V \psi_{j,k}^{(\ell)}\Big) = \psi_{k,k}^{(\ell)}.
\end{equation}
\end{corollary}
Putting the pieces together, we conclude that for any $f \in L^1(\R^3)$,
\begin{equation} \label{eq:Inverse1}
\begin{aligned}
(I+V\Res&(\lambda^2))^{-1}f \\ 
  = \ S&(\lambda)\tilde{Q}_0f \\ 
+ &\sum_{k,\ell} 
\big(\la f, \bar{\psi}_{1,k}^{(\ell)} \ra  - F_k^{(\ell)}(\lambda) \big)
\Big[ \Big(\sum_{1\le j \le k}
  \lambda^{-2(k+1-j)} V \psi_{j,k}^{(\ell)} \Big) + \psi_{k,k}^{(\ell)}\Big] 
\end{aligned}
\end{equation}
where $F_k^{(\ell)}(\lambda)$ stands for the inner product
$\big\la(I+V\Res(\lambda^2))S(\lambda)\tilde{Q}_0f, 
 \bar{\psi}_{1,k}^{(\ell)} \big\ra$.
The difference between $\la f, \bar{\psi}_{1,k}^{(\ell)}\ra$ and 
$F_k^{(\ell)}(\lambda)$ is always a bounded function of $\lambda$, hence the
operator inverse of $I + V\Res(\lambda^2)$ has an isolated singularity at
$\lambda = 0$.  The highest degree of blowup appears in each of the
$j=1$ terms and is potentially of order $\lambda^{-2k}$.

In fact we can be much more precise about the type of singular behavior of
$(I+V\Res(\lambda^2))^{-1}f$ in the vicinity of the origin. 
The key observation here is that for each $j \ge 2$,
\begin{align*}
\big\la f,\, \bar{\psi}_{j,k}^{(\ell)}\big\ra
&= \big\la \tilde{Q}_0 f, \bar{\psi}_{j,k}^{(\ell)}\big\ra \\
&= \big\la \tilde{Q}_0 (I+V\Res(\lambda^2))S(\lambda)\tilde{Q}_0 f,\,
\bar{\psi}_{j,k}^{(\ell)} \big\ra \\
&= \big\la (I+V\Res(\lambda^2))S(\lambda) \tilde{Q}_0 f,\,
  \bar{\psi}_{j,k}^{(\ell)} \big\ra.
\end{align*}

The middle equation is due to $S(\lambda)$ being the one-sided
inverse of $\tilde{Q}_0(I + V\Res(\lambda^2))$ and the others hold because
the adjoint of $\tilde{Q}_0$ acts as the identity on all functions orthogonal
to $X_{\rm diag}$.  It then follows that
\begin{align*}
F_k^{(\ell)}(\lambda) + &\sum_{2\le j\le k} \lambda^{2(j-1)}
  \big\la f,\, \bar{\psi}_{j,k}^{(\ell)} \big\ra \\
&= \ \Big\la (I + V\Res(\lambda^2))S(\lambda)\tilde{Q}_0 f,
 \sum_{1 \le j \le k} \lambda^{2(j-1)} \bar{\psi}_{j,k}^{(\ell)}
 \Big\ra \\
&= \ \Big\la S(\lambda)\tilde{Q}_0 f,\, (I + R_0^-(\lambda^2)\overline{V})
  \sum_{1 \le j \le k} \lambda^{2(j-1)}\bar{\psi}_{j,k}^{(\ell)}\Big\ra \\
&=\ -\lambda^{2k} \big\la S(\lambda)\tilde{Q}_0 f,\,
  R_0^-(\lambda^2)\bar{\psi}_{k,k}^{(\ell)} \big\ra
\end{align*}
with the last equation being a restatement of Corollary~\ref{cor:telescope}.
Substituting this expression in place of $F_k^{(\ell)}(\lambda)$ 
in~\eqref{eq:Inverse1} yields a formula for $(I+V\Res(\lambda^2))^{-1}$ that
isolates the coefficients of the pole at the origin.

\begin{lemma} \label{lem:Inverse}
Suppose $V$ satisfies the conditions of Theorem~\ref{thm:dispersive}. Then for
all $\lambda \in \R \setminus \{0\}$ sufficiently close to zero,
\begin{equation} \label{eq:Inverse2}
\begin{aligned}
(I+&V\Res(\lambda^2))^{-1}f \\
=\, &S(\lambda) \tilde{Q}_0 f
\\
&+ \sum_{k,\ell} \big\la S(\lambda)\tilde{Q}_0 f,\, R_0^-(\lambda^2)
   \bar\psi_{k,k}^{(\ell)} \big\ra  \Big[
    \Big( \sum_{1\le j \le k} \lambda^{2(j-1)} V\psi_{j,k}^{(\ell)} \Big)
     + \lambda^{2k} \psi_{k,k}^{(\ell)} \Big]
\\
&+ \sum_{k, \ell} \Big(\sum_{1 \le i \le k} \lambda^{2(i-1)} \big\la f,\,
  \bar{\psi}_{i,k}^{(\ell)}\big\ra \Big)
 \Big[ \Big( \sum_{1 \le j \le k} \lambda^{-2(k+1-j)} V\psi_{j,k}^{(\ell)}
          \Big) + \psi_{k,k}^{(\ell)} \Big]
\end{aligned}
\end{equation}
\end{lemma}
Note that the second line is a bounded function, and poles occur for each term
in the third line that has $i + j < k+2$.

\section{Proof of the Low-Energy Estimate} \label{sec:lowenergy}

We are now in a position to prove Theorem~\ref{thm:goal} in the low-energy
regime by using~\eqref{eq:Inverse2} to characterize of $T^+(\lambda)
= (I + V\Res(\lambda^2))^{-1}$ in terms of known functions.
\begin{theorem} \label{thm:lowenergy}
Suppose $V$ satisfies the conditions of Theorem~\ref{thm:dispersive}.
There exists $r > 0$ such that
\begin{equation} \label{eq:lowenergy}
\int_{\R^3}\int_\R \big|
\big[\chi(\,\cdot\,/r) T^+(\,\cdot\,)\big]^\wedge(\rho)f(x)\big|\,d\rho\,dx 
  \les \norm[f][1]
\end{equation}
holds for all $f \in L^1(\R^3)$ that are orthogonal to
$\X \subset L^\infty(\R^3)$.
\end{theorem}

\begin{proof}
This is a statement about the Fourier transform of
$(I + V\Res(\lambda^2))^{-1}$, so it makes sense to consider the Fourier
transform of each term in~\eqref{eq:Inverse2} individually.
The initial term $\chi(\lambda/r)S(\lambda)\tilde{Q}_0 f$ requires us to
retrace the construction of $S(\lambda)$ in~\eqref{eq:Slambda}.
Start with the power series
\begin{equation*}
\chi(\lambda/r)S(\lambda)\tilde{Q}_0 f = \sum_{m=0}^\infty
 (-1)^m\big(\chi(2\lambda/r)S_0\tilde{Q}_0 VB_0^+(\lambda^2)\big)^m
  \chi(\lambda/r)S_0 \tilde{Q}_0 f.
\end{equation*}
Aside from the restricted domain of $S_0$, and its associated
projection $\tilde{Q}_0$, this is the same series as~\eqref{eq:Neumann}
with the value $\lambda_0 = 0$.  The $L^1$ bounds on its Fourier transform
can be computed in an identical manner as well.

For functions $f \in L^1(\R^3)$ orthogonal to $\X$, each of the pairings
$\la f, \bar{\psi}_{j,k}^{(\ell)}\ra$ must be zero. 
Consequently the third line of~\eqref{eq:Inverse2} vanishes, taking all of
the most dangerous terms with it.

Finally, the second line can be controlled in its entirety provided the
Fourier transform (in $\lambda$) of each function
\begin{equation*}
\chi(\lambda/r)\big\la S(\lambda)\tilde{Q}_0 f,\,
   R_0^-(\lambda^2)\bar{\psi}_{k,k}^{(\ell)} \big\ra
\end{equation*}
is integrable, with its $L^1(\R)$ norm bounded by $\norm[f][1]$.  Each side of
the pairing is itself a function of $\lambda$, so its Fourier transform should
appear as the convolution
\begin{equation*}
\int_\R
  \big\la K_1(\,\cdot\,,\sigma)\overline{K_2(\,\cdot\,, \sigma-\rho)} \,d\sigma
= \int_\R \int_\R^3 K_1(x,\sigma) \overline{K_2(x,\sigma-\rho)}\,dx d\sigma
\end{equation*}
where $K_1(x,\rho)$ represents the Fourier transform of
$\chi(\lambda/r)S(\lambda)\tilde{Q}_0 f$ and $K_2(x,\rho)$ is the Fourier
transform of $\chi(2\lambda/r)R_0^-(\lambda^2)\bar{\psi}_{k,k}^{(\ell)}$.
It suffices to show that
\begin{equation}
\norm[K_1][L^1_{x,\rho}] \les \norm[f][1] \qquad {\rm and} \qquad
\norm[K_2][L^\infty_x L^1_\rho] \les 1
\end{equation}
and therefore
\begin{equation*}
\int_\R \Big| \int_\R \int_{\R^3} K_1(x,\sigma) K_2(x,\sigma-\rho)\,dx\,d\sigma
        \Big|\,d\rho \les \norm[f][1].
\end{equation*}

The estimate for $K_1$ has already been established by expanding out the
Neumann series for $S(\lambda)$.  The estimate for $K_2$ is straightforward,
using only the explicit resolvent kernel $R_0^-(\lambda^2){\scr(x,y)} =
e^{-i\lambda|x-y|}/(4\pi|x-y|)$.  This leads to the formula
\begin{equation*}
K_2(x,\rho) = \frac{r}{4} \int_{\R^3}
 \frac{\hat{\chi}(r(\rho+|x-y|)/2)}{|x-y|}\,\bar\psi_{k,k}^{(\ell)}(y)\,dy
\end{equation*}
Taking the $L^1(\R)$ norm in the $\rho$ variable and bring in absolute
values yields
\begin{equation*}
\norm[K_2(x,\,\cdot\,)][1] \le \norm[\hat{\chi}][1] \int_{\R^3}
   \frac{\big|\bar{\psi}_{k,k}^{(\ell)}(y)\big|}{2|x-y|}\,dy \les 1
\end{equation*}
because $\psi_{k,k}^{(\ell)}$ is assumed to belong to
$L^1(\R^3) \cap L^\infty(\R^3)$.
\end{proof}

\begin{remark}
The bounds on $\la S(\lambda)\tilde{Q}_0 f,\ R_0^-(\lambda^2)
\bar{\psi}_{k,k}^{(\ell)}\ra$ can be strengthened further.  Recall that the
range of $S_0$ could be any embedding of $L^1(\R^3)/VX_1$ into $L^1(\R^3)$,
and we chose it to be the subspace of functions orthogonal to 
$R_0^-(0)\X_{\rm diag}$.  By construction this subspace serves as the range
of $S(\lambda)$ as well.

Therefore one should obtain the same results by estimating the Fourier
transform of
$\chi(\lambda/r)\la S(\lambda)\tilde{Q}_0 f, 
(B_0^+(\lambda^2))^* \bar{\psi}_{k,k}^{(\ell)}\ra$.
Following the same method as above, but using the explicit kernel of
$B_0^+(\lambda^2)$ rather than $\Res(\lambda^2)$, it is possible to achieve
the bound $\norm[K_2][L^\infty_x L^1_\rho] \les r$.

This gives some evidence that our definition of $S_0$ leads to the ``best''
possible approximate inverse operator $S(\lambda)$.
\end{remark}

\nocite{BuPlStTa03}
\nocite{Du07}
\nocite{Ka65}
\nocite{LiLo01}
\nocite{Ra78}
\

\bibliographystyle{mrl}
\bibliography{MasterList}

\begin{thebibliography}{10}

\bibitem{Ag75}
S.~Agmon, \emph{{Spectral properties of Schr\"odinger operators and scattering
  theory}}, Ann.\ Sc. Norm.\ Super.\ Pisa. Cl.\ Sci. (4) \textbf{2} (1975),
  no.~2,  151--218.

\bibitem{Ar02}
W.~Arveson, {A Short Course on Spectral Theory}, Graduate Texts in Mathematics,
  209, Springer-Verlag (2002).

\bibitem{Be08}
M.~Beceanu, \emph{{A centre-stable manifold for the focussing cubic NLS in
  ${\bf R}^{1+3}$}}, Comm.\ Math.\ Phys. \textbf{280} (2008), no.~1,  145--205.

\bibitem{BuPlStTa03}
N.~Burq, F.~Planchon, J.~Stalker, and A.~S. Tahvildar-Zadeh, \emph{{Strichartz
  estimates for the wave and Schr\"odinger equations with the inverse--square
  potential}}, J.\ Funct.\ Anal. \textbf{203} (2003), no.~2,  519--549.

\bibitem{Cu01}
S.~Cuccagna, \emph{{Stabilization of solutions to nonlinear Schr\"odinger
  equations}}, Comm.\ Pure Appl.\ Math. \textbf{54} (2001), no.~9,  1110--1145.

\bibitem{Du07}
T.~Duyckaerts, \emph{{A singular critical potential for the Schr\"odinger
  operator}}, Canad.\ Math.\ Bull. \textbf{50} (2007), no.~1,  35--47.

\bibitem{ErSc04}
M.~B. Erdogan and W.~Schlag, \emph{{Dispersive estimates for Schr\"odinger
  operators in the presence of a resonance and/or an eigenvalue at zero energy
  in dimension three: I}}, Dyn.\ Partial Differ.\ Equ. \textbf{1} (2004),
  no.~4,  359--379.

\bibitem{Go08}
M.~Goldberg, \emph{{Strichartz estimates for the Schr\"odinger equation with
  time-periodic $L^{n/2}$ potentials}}, J.\ Funct.\ Anal. \hskip -.1cm. To
  appear.

\bibitem{Go06b}
---{}---{}---, \emph{{Dispersive bounds for the three-dimensional Schr\"odinger
  equation with almost critical potentials}}, Geom.\ and Funct.\ Anal.
  \textbf{16} (2006), no.~3,  517--536.

\bibitem{GoSc04}
M.~Goldberg and W.~Schlag, \emph{{Dispersive estimates for the Schr\"odinger
  operator in dimensions one and three}}, Comm.\ Math.\ Phys. \textbf{251}
  (2004), no.~1,  157--178.

\bibitem{IoJe03}
A.~Ionescu and D.~Jerison, \emph{{On the absence of positive eigenvalues of
  Schr\"odinger operators with rough potentials}}, Geom.\ and Funct.\ Anal.
  \textbf{13} (2003) 1029--1081.

\bibitem{JeKa79}
A.~Jensen and T.~Kato, \emph{{Spectral properties of Schr\"odinger operators
  and time-decay of the wave functions}}, Duke Math. J. \textbf{46} (1979),
  no.~3,  583--611.

\bibitem{JoSoSo91}
J.-L. Journ\'e, A.~Soffer, and C.~Sogge, \emph{{ Decay estimates for
  Schr\"odinger operators}}, Comm.\ Pure Appl.\ Math. \textbf{44} (1991),
  no.~5,  573--604.

\bibitem{Ka65}
T.~Kato, \emph{{Wave operators and similarity for some non-selfadjoint
  operators}}, Math.\ Ann. \textbf{162} (1965/1966) 258--279.

\bibitem{KeTa98}
M.~Keel and T.~Tao, \emph{{Endpoint Strichartz inequalities}}, Amer.\ J.\ Math.
  \textbf{120} (1998) 955--980.

\bibitem{KeRuSo87}
C.~Kenig, A.~Ruiz, and C.~Sogge, \emph{{Uniform Sobolev inequalities and unique
  continuation for second order constant coeffecient differential operators}},
  Duke Math.\ J. \textbf{55} (1987), no.~2,  329--347.

\bibitem{LiLo01}
E.~Lieb and M.~Loss, Analysis, Graduate Studies in Mathematics, 14, American
  Mathematical Society, Providence, second edition (2001).

\bibitem{Ra78}
J.~Rauch, \emph{{ Local decay of scattering solutions to Schr\"odinger's
  equation}}, Comm.\ Math.\ Phys. \textbf{61} (1978), no.~2,  149--168.

\bibitem{ReSi2}
M.~Reed and B.~Simon, {Methods of Modern Mathematical Physics. II. Fourier
  Analysis, Self Adjointness}, Academic Press [Harcourt Brace Jovanovich,
  Publishers], New York--London (1975).

\bibitem{ReSi1}
---{}---{}---, {Methods of Modern Mathematical Physics. I. Functional
  Analysis}, Academic Press, San Diego, revised and enlarged edition (1980).

\bibitem{RoSc04}
I.~Rodnianski and W.~Schlag, \emph{{Time decay for solutions of Schr\"odinger
  equations with rough and time-dependent potentials}}, Invent.\ Math.
  \textbf{155} (2004), no.~3,  451--513.

\bibitem{RoScSo05}
I.~Rodnianski, W.~Schlag, and A.~Soffer, \emph{{Dispersive analysis of charge
  transfer models}}, Comm.\ Pure Appl.\ Math. \textbf{58} (2005), no.~2,
  149--216.

\bibitem{SoWe04}
A.~Soffer and M.~Weinstein, \emph{{Selection of the ground state for nonlinear
  Schr\"odinger equations}}, Rev. Math.\ Phys. \textbf{16} (2004), no.~8,
  977--1071.

\bibitem{Ya95}
K.~Yajima, \emph{{The $W\sp {k,p}$-continuity of wave operators for
  Schr\"odinger operators}}, J. Math.\ Soc.\ Japan \textbf{47} (1995), no.~3,
  551--581.

\bibitem{Ya05}
---{}---{}---, \emph{{Dispersive estimate for SChr\"odinger equations with
  threshold resonance and eigenvalue}}, Comm.\ Math.\ Phys. \textbf{259}
  (2005), no.~2,  475--509.

\end{thebibliography}

\end{document}